\documentclass[11pt]{article}
\usepackage{amsmath, amsthm, amssymb,latexsym,enumerate}
\usepackage{hyperref}\usepackage{caption,subcaption,xcolor}
\usepackage{graphicx,cite}
\numberwithin{figure}{section}
\numberwithin{equation}{section}

\newtheorem{theorem}{Theorem}[section]
\newtheorem{lemma}[theorem]{Lemma}
\newtheorem{claim}[theorem]{Claim}

\newtheorem{conjecture}{Conjecture}

\theoremstyle{definition}
\newtheorem{remark}[theorem]{Remark}

\newcommand{\pc}{\mathrm{pc}_{\mathrm{opt}}}

\begin{document}
\baselineskip 18pt
\title{\bf\vspace{-1.5cm}  The optimal proper connection number of a graph with given independence number}

\author{Shinya Fujita\footnote{Supported by JSPS KAKENHI (No. 19K03603).}\\[1.5ex]
\small School of Data Science, Yokohama City University,\\
\small 22-2, Seto, Kanazawa-ku, Yokohama 236-0027, Japan\\
\small \tt fujita@yokohama-cu.ac.jp\\
[2.5ex]
Boram Park\footnote{Park's work was supported by Basic Science Research Program through the National Research Foundation of Korea (NRF) funded by the Ministry of Science, ICT and Future Planning (NRF-2018R1C1B6003577).}\\[1.5ex]
\small Department of Mathematics, Ajou University,\\
\small Suwon 16499, Republic of Korea \\
\small\tt borampark@ajou.ac.kr
}

\maketitle
\date

\begin{abstract} \normalsize
An edge-colored connected graph $G$ is \textit{properly connected} if between every pair of distinct vertices, there exists a path that no two adjacent edges have a same color.
Fujita (2019) introduced the \textit{optimal proper connection number} $\pc(G)$ for a monochromatic connected graph $G$,  to make a connected graph properly connected efficiently. More precisely, $\pc(G)$ is the smallest integer $p+q$ when one converts a given monochromatic graph $G$ into a properly connected graph by recoloring $p$ edges with $q$ colors.

In this paper, we show that $\pc(G)$ has an upper bound in terms of the independence number $\alpha(G)$. Namely,
we prove that for a connected graph $G$, $\pc(G)\le \frac{5\alpha(G)-1}{2}$. Moreover, for the case $\alpha(G)\leq 3$, we improve the upper bound to $4$, which is tight.\\

\noindent\textbf{Keywords:} Edge-colored graph, Properly connected graph, Optimal proper connection number, Independence number
\end{abstract}

\section{Introduction}\label{sec:Intro}

In this paper, we consider simple and finite graphs.
We refer the reader to \cite{West}
for terminology and notation not defined here.
For a graph $G$, let $\alpha (G)$ be the independence number of $G$ and $\alpha' (G)$ be the size of a maximum matching of $G$.
An edge-colored graph is \textit{properly colored} if no two adjacent edges have a same color. The notion on properly colored graphs sometimes plays an important role in a variety of fields related to the application of graphs. For example, properly colored paths and cycles appear in the context of some problems in genetics \cite{Dorninger, Dorninger2} and social sciences \cite{Chou}.
An edge-colored connected graph $G$ is \textit{properly connected} if between every pair of distinct vertices, there exists a properly colored path.
As described in \cite{Li}, making a connected graph $G$ properly connected is important in view of some real applications for building an efficient communication network.

Motivated by such applications, Borozan et al. \cite{Borozan} introduced a new notion called the \textit{proper connection number} $pc(G)$ of a connected graph $G$, that is, the minimum number of colors needed to color the edges of $G$ to make it properly connected. This notion attracts much attention on both theoretical and practical aspects in graph theory and a lot of work has been done so far in this area of study (see e.g. \cite{BD, F, GL, HLQ, HL, LW}).  Their work mainly explored the minimum number of colors to make $G$ properly connected, but most of them do not provide any information on how to make a connected graph properly connected efficiently. Recently, Fujita \cite{F2019} focused on this point of view and introduced the following notion.

In what follows, by a technical reason, we will firstly assume that a graph $G$ is always connected and all edges of $G$ are colored by color $1$ (thus, we say that $G$ is a \textit{monochromatic} graph). We then consider recoloring some of edges of $G$ so that $G$ is properly connected. Keeping this requirement in mind,
the \textit{optimal  proper  connection  number} of a monochromatic graph $G$, denoted
by $\pc(G)$, is defined as follows:
\begin{eqnarray*}
&& \pc(G)=  \min\left\{ p+q | \text{ By recoloring }p\text{ edges of }G\text{ with }q\text{ new colors}, \right. \\
&&\qquad \qquad  \qquad \qquad \qquad
 \left. G\text{ can be properly connected}\right\} . \end{eqnarray*}

By definition, it is obvious that $\alpha(G)=1$ if and only if $\pc(G)=0$; and moreover, $\pc(G)\neq 1$ holds for any graph $G$.

For integers $m\ge n\ge 2$ with $m+n\ge 9$, Fujita \cite{F2019} showed that $\pc(K_{m,n})=4$ for $n=2,3$ and $\pc(K_{m,n})=5$ for $n\ge 4$, and also showed that $\pc(T)=n-2-\alpha'(T)+\Delta(T)$ for any tree $T$ of order $n\ge 2$.
Very recently, Barish \cite{Barish} showed that computing $\pc(G)$ for a graph $G$ is NP-hard in general.

For graphs with small independence number, we have the following theorem.

\begin{theorem}[{\cite{F2019}}]\label{thm:alpha:2}
For a connected graph $G$ with $\alpha(G)=2$, $\pc(G)\le 3$ and the bound is tight.
\end{theorem}

In this paper, we generalize this result for graphs with a larger independence number. Indeed, we show that the optimal proper connection number of a graph $G$ has an upper bound in terms of the independence number of $G$ as follows.

\begin{theorem}\label{thm:main}
For a connected graph $G$, $\pc(G)\le \frac{5\alpha(G)-1}{2}$.
\end{theorem}

Along this line, we propose the following conjecture.

\begin{conjecture}\label{conj:general}
For a connected graph $G$ with $\alpha(G)\ge 3$, $\pc(G)\le 2\alpha(G)-2$.
\end{conjecture}

The bound on $\pc(G)$ is best possible if the conjecture is true. To see this, consider a star $K_{1,m}$ and note that $\pc(K_{1,m})=2m-2=2\alpha(K_{1,m})-2$.
Our second result is to show that the conjecture is true for connected graphs with independence number three.

\begin{theorem}\label{thm:alpha:3}
For a connected graph $G$ with $\alpha(G)=3$, $\pc(G)\le 4$.
\end{theorem}

In the proofs of our results, we utilize an extremal structure of graphs
with a fixed independence number.
A connected graph  $G$ is $\alpha$-\textit{minimal} if for every  $v\in V(G)$,
either $v$ is a cut-vertex of $G$ or $\alpha(G-v)<\alpha(G)$.
We prove that every $\alpha$-minimal graph $G$ has at most $2\alpha(G)-1$ vertices
(see Lemma~\ref{lem:folk}),
where the bound on $|V(G)|$ is best possible. To see this, consider the case $G$ is a path with odd number of vertices. As we will observe later, this lemma is very useful in our problem, and we believe that the lemma would also be useful for some other extremal problems in graphs with given independence number. As examples of such extremal problems, we briefly introduce some of them: We start with a famous structural result due to Chv\'atal and Erd\H{o}s  \cite{CE}. Their theorem states that any graph $G$ such that $\alpha(G)\le \kappa(G)$ contains a Hamiltonian cycle. The bound on $\alpha(G)$ is tight, but motivated by this theorem,
O et. al.\cite{O} gave the tight upper bound $\frac{\kappa(G)(|V(G)|+\alpha(G)-\kappa(G))}{\alpha(G)}$ on the order of longest cycles in a graph $G$. Aside from the connectivity of a graph, Egawa et. al. \cite{E2007} asked the maximum order $f(k,a)$ of a graph $G$ with $\alpha(G)\le a$ and with no $k$ vertex-disjoint cycles and they determined it for the case $1\le a\le 5$ or $1\le k\le 2$. Later, Fujita \cite{F2012} settled the case $k=3$ in this extremal problem. Along a slightly different line, Fujita et. al. \cite{F2018} investigated the least value $g(l,a)$ such that any graph $G$ of order at least $g(l,a)$ with $\alpha(G)\le a$ contains an $l$-connected subgraph of order at least $\frac{|V(G)|}{a}$ and they determined it for the cases $a=2,3$. Although we will not introduce further known results anymore, there are many directions to explore the extremal structure of graphs with given independence number. As we can learn from the above research progress, our standard approach to these extremal problems is to work on the partial problems on graphs with small independence number as a first step.
Doing this way, we might obtain some nice clue to settle the cases (ideally, all the cases) on graphs with a larger fixed independence number.
Indeed, Lemma~\ref{lem:folk} can be such a clue to some degree (see Remark~\ref{rmk1}).
We certainly present Theorem~\ref{thm:alpha:3} along the natural approach and it could contribute to providing some nice observation towards some other extremal problems on graphs $G$ with a fixed independence number.

The paper is organized as follows. Section 2 collects key observations for the proofs.
In Section 3, we prove of Theorem~\ref{thm:main}. We prove Theorem~\ref{thm:alpha:3} in Section 4.

\section{Preliminaries}

Here are some definitions and notation.
We denote $\{1,\ldots,n\}$ by $[n]$. Let $G$ be a graph.
For $A,B\subset V(G)$, $E_G(A,B)$ is the set of edges joining a vertex in $A$ and a vertex in $B$. For a vertex $x\in V(G)$, let $N_G(x)=\{y\in V(G)|\ xy\in E(G)\}$.
For $S\subset V(G)$, $G[S]$ and $G-S$ are the subgraph induced by $S$ and $V(G)-S$, respectively. If $S=\{v\}$, then we denote $G-S$ by $G-v$. For an edge $uw$ of $G$, $G-uw$ means the graph obtained from $G$ by deleting $uw$. For a subgraph $H$ of $G$, when $uv\in E(G)\setminus E(H)$ for two vertices $u,v$ of $H$, $H+uw$ means the graph obtained by adding $uw$ to $H$.

For a tree $T$ and  $x,y\in V(T)$, let $P_T(x,y)$ denote the path in $T$ from $x$ to $y$. For two vertex-disjoint paths $P$ and $Q$ of a graph $G$, if the last vertex of $P$ is adjacent to the first vertex of $Q$, then $P+Q$ denotes the path along $P$ and $Q$.

\begin{lemma}\label{lem:folk}
For an $\alpha$-minimal  graph $G$, $|V(G)|\le 2\alpha(G)-1$.
\end{lemma}

\begin{proof}
We use induction on $\alpha(G)$. If $\alpha(G)=1$ then $G$ is $K_1$ and so it is clear.
Suppose that the lemma holds for every $\alpha$-minimal connected graph with independence number less than $m$ for some $m\ge 2$.
Let $G$ be an $\alpha$-minimal connected graph with independence number $m$. Suppose to the contrary that  $|V(G)|\ge 2\alpha(G)$.
Let $I$ be a maximum independent set of $G$.
Take a vertex $x\in V(G)\setminus I$ so that the order of a minimum component of $G-x$ is as small as possible.
Let $G_1$, $G_2$, $\ldots$, $G_k$ ($k\ge 2$) be the components of $G-x$.
We may assume that  $|V(G_1)|\le |V(G_2)|\le \cdots \le |V(G_k)|$.
If $G_1$ contains a vertex $y\in  V(G)\setminus I$, then $G-y$ has a smaller component than $G_1$. To see this, note that the component of $G-y$ containing $x$ has at least $|V(G)\setminus V(G_1)|$ vertices, meaning that other components of $G-y$ have less than $|V(G_1)|$ vertices, respectively. However, this is a contradiction to the choice of $x$. Thus $V(G_1)=\{v\}$ for some vertex $v\in I$. Note that $G-v$ is connected.

\begin{claim}\label{claim:minimal}
If $|V(G_i)|\ge 2\alpha(G_i)$ for some $i\in \{2,\ldots,k\}$,
then there is a vertex $v_i\in V(G_i)$ such that $N_G(x)\cap V(G_i)=\{v_i\}$, $G_i-v$ is connected, $\alpha(G_i-v_i)=\alpha(G_i)$, and $\alpha(G-v_i)=\alpha(G)$.
\end{claim}

\begin{proof}
Suppose that there exists $i\in\{2,\ldots, k\}$ such that $|V(G_i)|\ge 2\alpha(G_i)$.
Then, by the induction hypothesis, $G_i$ is not $\alpha$-minimal.
Hence there exists a vertex $v_i\in V(G_i)$ such that $G_i-v_i$ is connected and $\alpha(G_i-v_i)=\alpha(G_i)$.
By taking a maximum independent set $I'_i$ of $G_i-v_i$,  $I'=(I-V(G_i)) \cup  I'_i$ is a maximum independent set of $G-v$, since $|I'|=|(I-V(G_i)) \cup  I'_i|=\sum_{i\in [k]} \alpha(G_i)=\alpha(G)$. Thus $\alpha(G-v_i)=\alpha(G)$.
By the $\alpha$-minimality of $G$, $G-v_i$ is disconnected, and so $N_G(x)\cap V(G_i)=\{v_i\}$. Therefore the claim holds.
\end{proof}

\begin{claim}\label{claim:2a-lem}
$|V(G_i)|=2\alpha(G_i)$ for every $i\in\{2,\ldots,k\}$.
\end{claim}
\begin{proof}
Suppose that $|V(G_i)|\ge 2\alpha(G_i)$ for some $i\in\{2,\ldots,k\}$.
By Claim~\ref{claim:minimal}, there is a vertex $v_i\in V(G_i)$ such that $N_G(x)\cap V(G_i)=\{v_i\}$, $\alpha(G_i-v_i)=\alpha(G_i)$, and $\alpha(G-v_i)=\alpha(G)$.
We will show that $G'_i=G_i-v_i$ is $\alpha$-minimal.
Take a vertex $y$ in $V(G'_i)$ such that  $G'_i-y$ is connected.
Since $N_G(x)\cap V(G_i)=\{v_i\}$, $G-y$ is connected. Then $\alpha(G-y)<\alpha(G)$  by  the $\alpha$-minimality of $G$.
Since $G-v_i$ has exactly two components $G-V(G_i)$ and $G_i-v_i$, we have
$$\alpha(G'_i-y) + \alpha(G-V(G_i)) \le \alpha(G-v_i-y)\le\alpha(G-y)<\alpha(G)=\alpha(G-v_i)=\alpha(G'_i) + \alpha(G-V(G_i)),$$
which follows that $\alpha(G'_i-y)<\alpha(G'_i)$.
Therefore, $G'_i$ is $\alpha$-minimal. By the induction hypothesis, $|V(G_i)|-1 = |V(G'_i)| \le 2\alpha(G'_i)-1 =2\alpha(G_i)-1$, and thus $|V(G_i)|\le 2\alpha(G_i)$.
Hence, we can conclude that $V(G_i)\le 2\alpha(G_i)$ for every $i\in\{2,\ldots,k\}$.
Then
\[ 2\alpha(G)-1\le |V(G)|-1 =\sum_{i\in [k]}|V(G_i)|\le 1+ \sum_{i=2}^{k}2\alpha(G_i) =2\alpha(G)-1, \]
which implies that $|V(G_i)|=2\alpha(G_i)$ for every $i\in\{2,\ldots,k\}$.
\end{proof}

By Claims~\ref{claim:minimal} and~\ref{claim:2a-lem},
for each $i\in\{2,\ldots,k\}$, there is a vertex $v_i\in V(G_i)$ such that $N_G(x)\cap V(G_i)=\{v_i\}$ and $\alpha(G_i-v_i)=\alpha(G_i)$.
We take a maximum independent set $I_i$ of $G_i-v_i$ for each $i\in\{2,\ldots,k\}$.
Then $I'=(\cup_{i=2}^{k}I_i)\cup \{x\}$ is an independent set of $G-v$.
Note that $|I'|=1+\sum_{i=2}^{k}\alpha(G_i-v_i)=\sum_{i\in [k]}\alpha(G_i)=\alpha(G)$, and so
$\alpha(G-v)=\alpha(G)$. Since $G-v$ is connected, we reach a contradiction to the $\alpha$-minimality of $G$.
\end{proof}

\begin{remark}\label{rmk1}
As a corollary of Lemma~\ref{lem:folk}, we have $\pc(G)\le 3\alpha(G)-2$.
To see this, by deleting vertices of $G$ one by one as long as the resulting graph is connected and has independence number $\alpha(G)$,
we find an induced connected subgraph $H$ of $G$ such that $\alpha(H)=\alpha(G)$ and $H$ is $\alpha$-minimal.
Take a spanning tree $T$ of $H$ and then recolor all the edges of $T$ with $\Delta(T)$ new colors so that $T$ is properly connected. From the fact that $\alpha(G[V(T)])= \alpha(G)$,
every vertex in $V(G)\setminus V(T)$ has a neighbor in $V(T)$  and every two nonadjacent vertices $x,y$ in $V(G)\setminus V(T)$ have distinct neighbor in $V(T)$.
Thus by taking paths in $T$, one can check that $G$ is properly connected. Hence,
\[\pc(G)\le \Delta(T)+|E(T)| \le \alpha(G)+(2\alpha(G)-2) = 3\alpha(G)-2.\]
It is not enough to obtain Theorem~\ref{thm:main} or Theorem~\ref{thm:alpha:3}, but the   idea of this argument plays a key role in its proof.
\end{remark}

\begin{lemma}\label{lem:a:minimal}
For a connected graph $G$, there is a connected induced subgraph $H$ of $G$ such that $\alpha(H)=\alpha(G)$ and $|V(H)|\le 2\alpha(G)-1$.
If such $H$ has minimum number of vertices, then for every spanning tree $T$ of $H$,
every pendent vertex of $T$ belongs to a maximum independent set of $H$ and so $T$ has at most $\alpha(G)$ pendent vertices.
\end{lemma}

\begin{proof}
By deleting vertices of $G$ one by one as long as the resulting graph is connected and has independence number $\alpha(G)$, we can find a connected induced subgraph $H$ of $G$
such that $\alpha(H)=\alpha(G)$.
By lemma~\ref{lem:folk}, $|V(H)|\le 2\alpha(H)-1=2\alpha(G)-1$.

Take such induced subgraph $H$ with minimum number of vertices. Take an independent set $I$ of $G$ such that $I\subset V(H)$.
If a spanning tree $T$ of $H$ has a pendent vertex $x$ not in $I$, then $H-x$ is a connected induced subgraph of $G$ such that $\alpha(H-x)=\alpha(G)$, which contradicts to the minimality of $|V(H)|$.
Thus for every spanning tree $T$ of $H$, every pendent vertex of $T$ belongs to $I$, and so $T$ has at most $\alpha(G)$ pendent vertices.
\end{proof}

\begin{lemma}\label{lem:unique:max}
Let $T$ be a tree with at least two vertices and at most $p$ pendent vertices.
If $\Delta(T)>\frac{p+3}{2}$, then there exists a unique vertex of maximum degree and a vertex with second largest degree has degree at most $\Delta(T)-2$.
\end{lemma}

\begin{proof}For simplicity, let $d=\Delta(T)$.
Suppose that $d>\frac{p+3}{2}$. Note that since $p\ge 2$, $d\ge 3$.
Let $v$ be a vertex of maximum degree $d$, and $w$ be a vertex with largest degree among vertices in $V(T)\setminus\{v\}$.
Suppose to the contrary that $\deg_T(w)\ge d-1$.
Let $e$ be the first edge on the path $P_T(v,w)$.
Let $T_v$ (resp. $T_w$) be the component of $T-e$ containing $v$ (resp. $w$).
Then $\Delta(T_v)\ge d-1 (\ge 2)$ and so $T_v$ has at least $d-1$ pendent vertices and those are also pendent vertices of $T$.
If $e\neq vw$, then $\Delta(T_w)=\deg_T(w)$ and so $T_w$ has at least $d-1$ pendent vertices and those except at most one are also pendent vertices of $T$.
If $e=vw$, then $T_w$ has at least $d-2$ pendent vertices and those are also pendent vertices of $T$. Hence, there are at least $(d-1)+(d-2)$ pendent vertices in $G$.
Thus $2d-3 \le p$ or $d\le \frac{p+3}{2}$, a contradiction.
\end{proof}

\begin{lemma}\label{lem:tree:components}
Let $T$ be a spanning tree of a connected graph $H$ such that $\Delta(T)$ is small as possible.
If $T$ has a unique vertex $v$ of maximum degree and  the second largest degree of $T$ is at most $\Delta(T)-2$, then $H-v$ has $\Delta(T)$ components.
\end{lemma}

\begin{proof}
Take a component $H'$ of $H-v$. It is enough to show that  $|E_T(v,V(H'))| = 1$.
It is also trivial that $|E_T(v,V(H'))|\ge 1$ since $T$ is a spanning tree of $H$.
Suppose to the contrary that $|E_T(v,V(H'))|\ge 2$.
Let $E_T(v,V(H'))=\{vx_1,\ldots,vx_k\}$, and $C_i$ be the component of $T-v$ containing $x_i$.
Since $H'$ is connected, $E_{H'}(V(C_i),V(C_j))\neq\emptyset$ for some distinct $i,j$.
Then $T'=(T-vx_i)+e$ for some $e\in E_{H'}(V(C_i),V(C_j))$ is also a spanning tree of $H$.
Since  the second largest degree of $T$ is at most $\Delta(T)-2$, we have $\Delta(T')\le \Delta(T)-1$, which is a contradiction to the minimality of $\Delta(T)$.
\end{proof}

%\subsection{A key lemma}

The following lemma contains a key idea for the proof of Theorem~\ref{thm:main}.

\begin{lemma}\label{lem:key}
Let $G$ be a connected graph with $\alpha(G)\ge 3$, and $H$ be a minimum connected induced subgraph of $G$ such that $\alpha(H)=\alpha(G)$.
For every spanning tree $T$ of $H$,
\[\pc(G)\le |E(T)|+\Delta(T)-|M|,\]
where  $M$ is a set of edges $ww'\in E(T)$ such that
$\deg_T(w)=1$, $\deg_T(w')=2$, and the following property (\S) holds:
\begin{quote}
(\S) For every $x\in V(G)\setminus V(T)$, if $xw\in E(G)$ then  $xw'\in E(G)$.
\end{quote}
\end{lemma}

\begin{proof}
Let $I$ be a maximum independent set of $H$.
We color all edges in $E(T)- M$ with the new colors $2,3,\ldots,\Delta(T)+1$ (this is possible since the edge chromatic number of a tree $T$ is $\Delta(T)$). Since the number of new colors is $\Delta(T)$ and the number of  edges colored by new colors is $|E(T)|-|M|$, it is enough to show that $G$ is properly connected.

Let $W$ be the set of pendent vertices of $T$ incident to an edge of $M$.
For every $u\in V(G)\setminus V(H)$,
we define \[ X(u)=(N_G(u)\cap I)-W \quad\text{and}\quad  X'(u)=(N_G(u)\cap V(H))-W .\]
By definition, it is clear that $X(u)\subset X'(u)\subset V(H)$ for every $u\in V(G)\setminus V(H)$.

\begin{claim}\label{claim:Xv}
For every $u\in V(G)\setminus V(H)$,  $X'(u) \neq \emptyset$.
\end{claim}

\begin{proof} Take $u\in V(G)\setminus V(H)$.
It is trivial that if $X(u)\neq \emptyset$ then $X'(u)\neq \emptyset$.
Suppose that $X(u)=\emptyset$.
Since $u\not\in I$, $N_G(u)\cap I\neq \emptyset$ by the maximality of $|I|$.
Thus by the definition of $X(u)$,  $(N_G(u)\cap I)\cap W \neq \emptyset$.
Take $w\in(N_G(u)\cap I)\cap W$. Then $ww'\in M$ for some vertex $w'\in V(H)$. By the property ($\S$),  $X'(u)$ contains $w'$ and therefore $X'(u)\neq \emptyset$.
\end{proof}

Take two vertices $u,v$ of $G$ that are not adjacent.
If $u,v\in V(H)$, then we can easily check that every path $P_T(u,v)$ is properly colored by the definition of $M$ (since $\deg_G(w')=2$ and $\alpha(G)\ge 3$, we know that $M$ is a matching).
Suppose that $u\not\in V(H)$ and $v\in V(H)$.
Then we can take a vertex $u'\in X'(u)$ by Claim~\ref{claim:Xv}.
Note that the edge $uu'$ has the color 1.
If the first edge of the path $P_T(u',v)$ does not use the color 1, then $u + P_T(u',v)$ gives a properly colored path from $u$ to $v$.
Suppose that the first edge of the path $P_T(u',v)$ is colored by the color 1.
Then $u'v\in M$ and $X'(u)=\{u'\}$. Note that by Lemma~\ref{lem:a:minimal}, $v\in I$ and so $u'\not\in I$.
Since $I\cup\{u\}$ is not an independent set of $G$ by the maximality of $|I|$, $u$ has a neighbor $z$ in $I$ and $z\in W$. Then $z\neq u'$ and  $zz'\in E(T)$ for some $z'\in V(T)$. By ($\S$), it follows that $uz'\in E(G)$ and so $z'\in X'(u)$. Since $u\neq z$, $u'\neq z'$  which is a contradiction.

Now suppose that $u,v\not\in V(H)$.
If we can take $u'\in X'(u)$ and $v'\in X'(v)$ so that $u'\neq v'$, then
$u+P_T(u',v')+v$ gives a properly colored path from $u$ to $v$.
Suppose that we cannot take distinct  $u'\in X'(u)$ and $v'\in X'(v)$.
Claim~\ref{claim:Xv} implies that $X'(u)=X'(v)=\{z\}$ for some vertex $z\in V(H)\setminus W$.
If $ (N_G(u)\cap I)\cap W=(N_G(v)\cap I)\cap W=\emptyset$,
then $N_G(u)\cap I=N_G(v)\cap I=\{z\}$, which means that $(I\setminus\{z\})\cup\{u,v\}$ is an independent set, a contradiction to the maximality of $|I|$.
We may assume that $(N_G(u)\cap I)\cap W\neq \emptyset$.
Take  $w\in (N_G(u)\cap I)\cap W$.
Then $ww'\in M$ for some vertex $w'\in V(H)\setminus W$.
By the property ($\S$), $uw'\in E(G)$.
Hence $X'(u)=X'(v)=\{z\}=\{w'\}$.
Note that since $w\in I$, we have $z\not\in I$.

Since $N_G(v)\cap I\neq\emptyset$ by the  maximality of $|I|$, there is $y\in (N_G(v)\cap I)\cap W$. By definition, $yy'\in M$ and so $y'\in X'(v)$ by the property ($\S$) for some vertex $y'\in V(H)\setminus W$.
Thus $y'=z=w'$, $y=w$, and so $N_G(u)\cap V(H) = N_G(v)\cap V(H)=\{w,w'\}$.
Since $u$ and $v$ are not adjacent, $(I-\{w\})\cup\{u,v\}$ is an independent set of $G$, which is a contradiction to the maximality of $|I|$.
\end{proof}

\section{Proof of Theorem~\ref{thm:main}}

\begin{proof}[Proof of Theorem~\ref{thm:main}]
Recall Theorem~\ref{thm:alpha:2} and the fact that $\alpha(G)=1$ if and only if $\pc(G)=0$. Thus the theorem holds when $\alpha(G)\le 2$.
Let $G$ be a connected graph with $\alpha(G)\ge 3$.
We take a connected  induced subgraph $H$ of $G$ with $\alpha(H)=\alpha(G)$, a spanning tree $T$ of $H$, a maximum independent set $I$ of $H$ so that
\begin{itemize}
\item[(1)] $|V(H)|$ is minimum,
\item[(2)] $\Delta(T)$ is minimum subject to (1), and
\item[(3)] $|I\cap \{x\in V(H)|\deg_T(x)=\Delta(T)\}|$ is minimum subject to (2).
\end{itemize}
By Lemma~\ref{lem:a:minimal} and by the choice of $H$, $T$ and $I$, it is easy to see that the following hold:
\begin{itemize}
\item $|V(H)|\le 2\alpha(G)-1$,
\item $T$ is a tree at most $\alpha(G)$ pendent vertices, and
\item $I$ is a maximum independent set $I$ of $G$.
\end{itemize}
If $\Delta(T)\le \frac{\alpha(G)+3}{2}$, then by Lemma~\ref{lem:key},
\[\pc(G)\le  |E(T)|+\Delta(T)\le (2\alpha(G)-2) + \frac{\alpha(G)+3}{2} \le \frac{5\alpha(G)-1}{2}.\]
In the following, we suppose that $\Delta(T)>\frac{\alpha(G)+3}{2}$.
By Lemma~\ref{lem:unique:max}, there is a unique vertex $v$ in $T$ with maximum degree and so the condition (3) for the choice of $I$ is equivalent to say `$|I\cap \{v\}|$  is minimum'.
We also note that by Lemma~\ref{lem:unique:max},  the second largest degree of $T$ is at most $\Delta(T)-2$. We let $d=\Delta(T)$. By Lemma~\ref{lem:tree:components}, $H-v$ has exactly $d$ components. Let $H _1$, $\ldots$, $H_{d}$ be the components of $H-v$, and let $H^*_i=G[V(H_i)\cup\{v\}]$ for every $i\in[d]$.

\begin{claim}\label{claim:notv}
For every $i\in[d]$, $\alpha(H_i)=|I\cap V(H_i)|$. Moreover, if $v\in I$ then $\alpha(H^*_i) =|I\cap V(H^*_i)|= \alpha(H_i)+1$ and every maximum independent set of $H^*_i$ contains $v$.
\end{claim}
\begin{proof}
It is clear that $|I\cap V(H_i)|\le \alpha(H_i)$ for every $i\in[d]$.
Since the union of independent sets of $H_i$'s is also an independent set of $G$, $\sum_{i\in[d]} \alpha(H_i)  \le \alpha(G).$
If $v\not\in I$, then the claim holds clearly, since
\[\alpha(G)=|I|= \sum_{i\in[d]} |I\cap V(H_i) |\le \sum_{i\in[d]} \alpha(H_i)  \le \alpha(G).\]
Suppose that $v\in I$.
Then it holds that
\[ \alpha(G)-1 =|I|-1 = \sum_{i\in[d]} |I\cap V(H_i) |\le \sum_{i\in[d]} \alpha(H_i)  \le \alpha(G).\]
Suppose to the contrary that there is some $i\in[d]$ such that  $\alpha(H_i) > |I\cap V(H_i)|$.
Then $\alpha(H_i)=|I\cap V(H_i)|+1=|I\cap V(H^*_i)|$ and $\alpha(H_j)=|I\cap V(H_j)|$ for every
$j\in[d]$ with $j\neq i$.
Take a maximum independent set $I'_i$ of $H_i$.
Then  $I'=I'_i \cup \left(I-V(H^*_i)\right)$ is an independent set of $G$.
Since $|I'|=|I'_i|+ |I|-|I\cap V(H^*_i)|=|I|$,  $I'$ is a maximum independent set of $G$.
Then  $I'\subset V(H)$ and  $v\not\in I'$, which is a contradiction to the condition (3) for the choice of $I$. Hence, $\alpha(H_i)=|I\cap V(H_i)|$.
This also implies that  $\alpha(H_i)+1\ge \alpha(H^*_i) \ge |I\cap V(H^*_i)|= \alpha(H_i)+1$. Thus
$\alpha(H_i)+1 = \alpha(H^*_i)$ and  every maximum independent set of $H^*_i$ contains $v$.
\end{proof}

\begin{claim}\label{claim:2a}
For every $i\in [d]$, $|V(H_i)|\le 2\alpha(H_i)$.
\end{claim}

\begin{proof}
Suppose to the contrary that  $|V(H_i)|\ge 2\alpha(H_i)+1$ for some $i\in [d]$.
First, suppose that $v\not \in I$. We take a minimal connected induced subgraph $H'_i$ of $H_i$ such  that $\alpha(H'_i)=\alpha(H_i)$.
By Lemma~\ref{lem:a:minimal},  $|V(H'_i)|\le 2\alpha(H_i)-1$.
It is enough to show that $|V(H_i)-V(H'_i)|\le 1$. For simplicity, let $X=V(H_i)-V(H'_i)$. Suppose to the contrary that $|X|\ge 2$.
If $V(H_i)$ contains a neighbor of $v$ then $H'=H-X$ is a connected induced subgraph of $G$ such that $\alpha(H')=\alpha(G)$ and $|V(H')|<|V(H)|$, which is a contradiction to the condition (1) for the choice of $H$.
Thus $V(H_i)$ does not contain a neighbor of $v$.
Then there is an edge $xv$ of $G$ for some $x\in X$.
Since $H'_i$ contains a maximum independent set $I'$ of $H_i$, $E_G(x, V(H'_i) )\neq \emptyset$ by the maximality of $|I'|$.
Then $H':=H- (X\setminus \{x\})$ is a connected induced subgraph of $G$ such that $\alpha(H')=\alpha(G)$ and $|V(H')|<|V(H)|$,  which is a contradiction to the condition (1) for the choice of $H$.

Secondly, suppose that $v\in I$. We take a minimal connected induced subgraph $H'_i$ of $H^*_i$ such that $\alpha(H'_i)=\alpha(H^*_i)$.
By Lemma~\ref{lem:a:minimal}~and Claim~\ref{claim:notv}, $|V(H'_i)|\le 2(\alpha(H^*_i)+1)-1=2\alpha(H_i)+1$.
By Claim~\ref{claim:notv} again, $v$ belongs to a maximum independent set of $H'_i$ and so $v\in V(H'_i)$.
Consider $H'=H-X$ where $X=V(H_i)-V(H'_i)$. Then $\alpha(H')=\alpha(H'_i)+|I-V(H^*_i)|=\alpha(H)$.
If $X\neq \emptyset$ then $|V(H')|<|V(H)|$, which is a contradiction to the condition (1) for the choice of $H$. Thus $X=\emptyset$ and so $|V(H_i)|\le |V(H'_i)|-1\le 2\alpha(H_i)$.
\end{proof}

Let $\mathcal{S}_2=\{H_i\mid |V(H_i)|=2\}$  and $\mathcal{S}_4=\{H_i\mid |V(H_i)|\ge 4 \}$.
Then \[  1+(d-|\mathcal{S}_4|-|\mathcal{S}_2|) +2|\mathcal{S}_2|+ 4|\mathcal{S}_4| \le |V(H)| \le  2\alpha-1,\]
and so  $|\mathcal{S}_2|+3|\mathcal{S}_4| \le 2\alpha -d-2$.
Then \begin{eqnarray}\label{eq:S4}
&&|\mathcal{S}_4| \le \frac{2\alpha -d -2-|\mathcal{S}_2|}{3}.
\end{eqnarray}
Moreover, if $H_i\not\in\mathcal{S}_2\cup \mathcal{S}_4$, then $|V(H_i)|=1$ or $3$ and so $|V(H_i)|\le 2\alpha(H_i)-1$.
Hence, by Claim~\ref{claim:2a},
\[|V(T)|\le 1 +\sum_{i\in [d]} (2\alpha(H_i)-1)  + |\mathcal{S}_2|+|\mathcal{S}_4|\le 1+ 2\alpha(G)-d+|\mathcal{S}_2|+|\mathcal{S}_4|.\]
By applying Lemma~\ref{lem:key},
\[\pc(G) \le (2\alpha(G)-d+|\mathcal{S}_2|+|\mathcal{S}_4|) +d-|M|= 2\alpha(G) +|\mathcal{S}_2|+|\mathcal{S}_4| -|M|,
\]
where $M$ is a set of edges $T$ in $G$ satisfying the properties in Lemma~\ref{lem:key}.
By \eqref{eq:S4},
\[\pc(G)
\le  2\alpha(G) +|\mathcal{S}_2|+\frac{2\alpha(G) -d  -2-|\mathcal{S}_2|}{3}-|M|< \frac{5\alpha(G)-1}{2}  +\frac{2|\mathcal{S}_2|-2}{3}-|M|,\]
where the last inequality is from the assumption that  $d>\frac{\alpha(G)+3}{2}$.

Suppose that $|M|\ge |\mathcal{S}_2|-1$.
Then\[
\pc(G) \le \frac{5\alpha(G)-1}{2}  +\frac{2|\mathcal{S}_2|-2}{3} - (|\mathcal{S}_2|-1) \le  \frac{5\alpha(G)-1}{2}  +\frac{-|\mathcal{S}_2|+1}{3}\le \frac{5\alpha(G)-1}{2}  +\frac{1}{3}.\]
Since $\pc(G)$ is an integer, $\pc(G)\le \frac{5\alpha(G)-1}{2}$, and so the theorem holds.
Thus, it is enough to show that $|M|\ge |\mathcal{S}_2|-1$.
Precisely, we will show that  if $|\mathcal{S}_2|\ge 2$, then every edge of $\mathcal{S}_2$ satisfies the properties in Lemma~\ref{lem:key}.

Suppose that $|\mathcal{S}_2|\ge 2$. Since $\mathcal{S}_2\neq \emptyset$, the minimality of $|V(H)|$ implies that $v\in I$.
Take $H_i \in \mathcal{S}_2$. Let $V(H_i)=\{w,w'\}$, and  without loss of generality, we may assume that $\deg_T(w)=1$ and $\deg_T(w')=2$.
Suppose to the contrary that there is a vertex $x\in V(G)\setminus V(H)$ with $xw\in E(G)$ and $xw'\not\in E(G)$. We first show that
\begin{eqnarray}
&&N_G(x)\cap V(H) \subset \{v,w,w'\}.  \label{eq:vww'}
\end{eqnarray}

To show \eqref{eq:vww'} by contradiction,
suppose that there is $y\in N_G(x)\cap V(H_j)$ for some $j$ with  $j\neq i$.
Consider the subgraph $H'$ of $G$ induced by $(V(H)-\{w'\})\cup \{x\}$.
Let $T'$ be a tree obtained from $T-w'$ by adding the vertex $x$ and adding edges $xw,xy$.
Then $H'$ has $I$ as an independent set (and therefore $\alpha(H')=\alpha(G)$) and  $T'$ as a spanning tree.
Since $|V(H')|=|V(H)|$  and  $\Delta(T')\le \Delta(T)-1$, we reach a contradiction to the condition (2) for the choice of $T$.
Thus, $ (N_G(x)\cap V(H_j)) =\emptyset$ for every $j\in[d]$ with  $j\neq i$, which also implies \eqref{eq:vww'}.

Since $|\mathcal{S}_2|\ge 2$, we can take some $H_{i'}\in \mathcal{S}_2$ where $i\neq i'$.
Let $V(H_{i'})=\{z,z'\}$, where $\deg_T(z)=1$ and $\deg_T(z')=2$.
Consider the subgraph $H'$ of $G$ induced by $(V(H)-\{z\})\cup \{x\}$.
Let $T'$ be a tree obtained from $T-z$ by adding the vertex $x$ and adding edges $xw$. Then $v$ is a unique vertex of maximum degree in $T'$ and $\Delta(T')=\Delta(T)$.
By \eqref{eq:vww'}, $I'=(I-\{w,z,v\})\cup \{x,w',z'\}$ is a maximum independent set of $G$ such that $I'\subset V(H')$.
Since  $v\not\in I'$, we reach a contradiction to the condition (3) for the choice of $I$.
\end{proof}

\section{Proof of Theorem~\ref{thm:alpha:3}}
\begin{lemma}\label{prop:config}
For a connected graph $G$ with $\alpha(G)=3$, $\pc(G)\le 4$ if one of the following holds:
\begin{itemize}
\item[\rm(i)] $V(G)$ can be partitioned into three cliques;
\item[\rm(ii)] $V(G)$ can be partitioned into four cliques and there is a matching $M$ of size three such that those four disjoint cliques together with the edges in $M$ is a connected spanning subgraph of $G$;
\item[\rm(iii)] $G$ has a $6$-cycle $x_1,y_1,x_2,y_2,x_3,y_3,x_1$ such that each of $\{x_1,x_2,x_3\}$ and  $\{y_1,y_2,y_3\}$ is an independent set of $G$;
\item[\rm(iv)] $G$ has one of $H_1$-$H_6$ in Figure~\ref{fig:config} as a subgraph such that $\{x_1,x_2,x_3\}$ is an independent set of $G$.
\end{itemize}
\end{lemma}

\begin{figure}[h!]
  \centering
  \includegraphics[width=16cm]{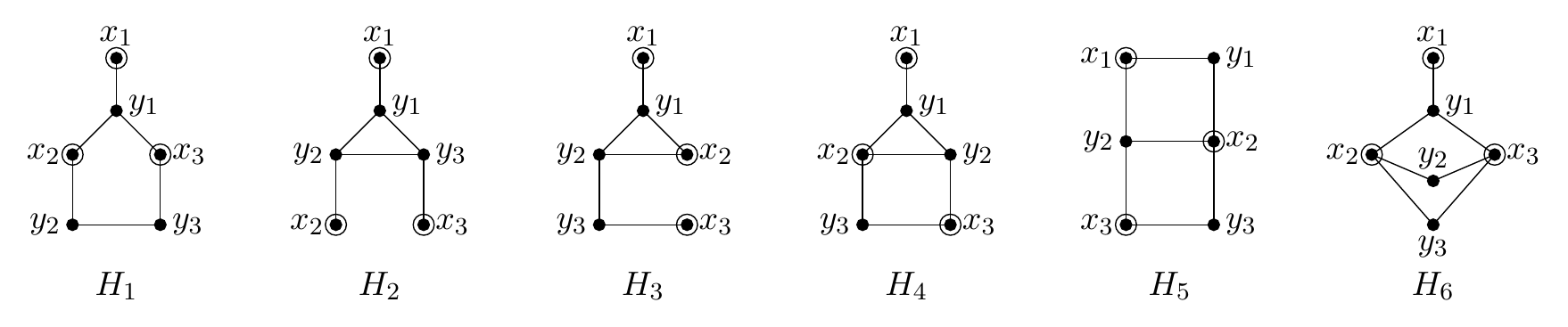}
  \caption{Subgraphs $H_1$-$H_6$ of $G$, where the circled vertices $x_1,x_2,x_3$ form an independent set of $G$}\label{fig:config}
\end{figure}

\begin{proof}
If (i) holds, that is, $V(G)$ is partitioned into three cliques $X_1$, $X_2$, $X_3$, then we may assume that there are two edges $e\in E_G(X_1,X_2)$ and $e'\in E_G(X_1,X_3)$, and coloring $e$ and $e'$ with new colors 2 and 3, respectively, gives a properly colored path between every two vertices of $G$. Thus in this case, $\pc(G)\le 2+2=4$.
If (ii) holds, then coloring the edges in $M$ by the new color 2 gives a properly colored path between every two vertices of $G$, and so $\pc(G)\le 1+3=4$.
If (iii) holds, then since  $\{x_1,x_2,x_3\}$  and $\{y_1,y_2,y_3\}$ are maximum independent sets in $G$,
every vertex in $V(G)\setminus \{x_1,x_2,x_3,y_1,y_2,y_3\}$ has a neighbor in $\{x_1,x_2,x_3\}$ and a neighbor in $\{y_1,y_2,y_3\}$, and so
coloring three edges $x_1y_1,x_2y_2,x_3y_3$ by the new color 2 gives a properly colored path between every two vertices of $G$, and so $\pc(G)\le 1+3=4$. Hence if (i), (ii), or (iii) holds, then $\pc(G)\le 4$.

Now suppose to the contrary that (iv) holds and $\pc(G)>4$.
Then $G$ has $H$ as a subgraph which is isomorphic to $H_i$ for some $i\in[6]$, where $I=\{x_1,x_2,x_3\}$ is an independent set of $G$. We label the vertices depicted in Figure~\ref{fig:config}.
We color three edges $x_1y_1$, $x_2y_2$, $x_3y_3$ by the new color $2$.
Since $\pc(G)>4$, there are two nonadjacent vertices $w$ and $w'$ of $G$ which can not be connected by a properly colored path. Since there is a properly colored path between any pair of two vertices in $H$, we may assume that $w\not\in V(H)$.

First, suppose that  $H$ is isomorphic to $H_i$ for some $i\in[3]$.
Since $wx_k\in E(G)$ for some $k\in[3]$ by the maximality of $|I|$, there is a path from $w$ to every vertex of $H$ starting with $w,x_k,y_k$.
Thus $w'\not\in V(H)$. Since $\alpha(G)=3$, we may assume that $wx_k,w'x_j\in E(G)$
for two distinct $k,j\in\{1,2,3\}$. Then we can find a properly colored path starting with $w,x_k,y_k$ and ending with $y_j,x_j,w'$. Consequently, there is a properly colored path connecting $w$ and $w'$, a contradiction.
Suppose that $H$ is isomorphic to $H_4$.
From two properly colored paths $x_1,y_1,x_2,y_2,x_3,y_3$ and $x_1,y_1,y_2,x_2$,
we know that $w'\not\in V(H)$ and  each of $w$ and $w'$ is not adjacent to $x_1$.
Thus $I'=\{x_1,w,w'\}$ is an independent set of $G$.
Without loss of generality, we may assume that $wx_2,w'x_3\in E(G)$.
Then $y_3$ has a neighbor in $I'$ by maximality of $|I'|$.
If $y_3w\in E(G)$, then $w,y_3,x_3,w'$ is a properly colored path between $w$ and $w'$.
If $y_3w'\in E(G)$, then $w,x_2,y_2,x_3,y_3,w'$ is a properly colored path between $w$ and $w'$.
If $y_3x_1\in E(G)$, then  $w,x_2,y_2,y_1,x_1,y_3,x_3,w'$ is a properly colored path between $w$ and $w'$.
Thus, we reach a contradiction.

Suppose that $H$ is isomorphic to $H_5$. By the condition (iii),
$\{y_1,y_2,y_3\}$ is not an independent set of $G$. If $y_1y_3\in E(G)$ then $G$ has  a subgraph isomorphic to $H_1$ with independent set $I$ ($x_1y_2$ of $H_5$ plays a role of $x_1y_1$ of $H_1$), and
if $y_1y_2\in E(G)$ or $y_2y_3\in E(G)$ then $G$ has a subgraph isomorphic to $H_4$ with independent set $I$, which implies that we can reach a same contradiction.

Suppose that $H$  is isomorphic to $H_6$.
If  $y_2y_3\in E(G)$ then $G$ has  a subgraph isomorphic to $H_1$  with independent set $I$, and
if $y_1y_2\in E(G)$ or $y_1y_3\in E(G)$ then $G$ has  a subgraph isomorphic to $H_4$ with independent set $I$, a contradiction.
Thus $\{y_1,y_2,y_3\}$ is an independent set of $G$.
From the fact that $w$ has a neighbor in $\{x_1,x_2,x_3\}$ and a neighbor in $\{y_1,y_2,y_3\}$, we can find a properly colored path from $w$ to every vertex of $V(H)$.
Thus, $w'\not\in V(H)$.

If $wx_1\in E(G)$, then using the path $w,x_1,y_1,x_2,y_2,x_3,y_3$ and edge $w'y_t$ for some $t\in[3]$, we can find a properly colored path between $w$ and $w'$, a contradiction.
Thus $wx_1\not\in E(G)$. Similarly, $w'x_1\not\in E(G)$.
Then $I'=\{w,w',x_1\}$ is an independent set of $G$.
We may assume that $wx_2,w'x_3\in E(G)$.
If $y_2x_1,y_3x_1\in E(G)$, then $y_2,x_1,y_3,x_3,y_1,x_2$ is a 6-cycle satisfying (iii), a contradiction.
Thus,  $y_2x_1\not\in E(G)$ or $y_3x_1\not\in E(G)$.
Since each of $y_2$ and $y_3$ has a neighbor in $I'$ by maximality of $I'$, we reach a contradiction by one of the following four observations: (1) If $y_2w\in E(G)$, then $w,y_2,x_2,y_3,x_3,w'$ is a properly colored path; (2) if $y_2w'\in E(G)$, then $w,x_2,y_2,w'$ is a properly colored path; (3) if $y_3w\in E(G)$, then $w,y_3,x_3,w'$ is a properly colored path;
(4) if $y_3w'\in E(G)$, then $w,x_2,y_2,x_3,y_3,w'$ is a properly colored path.
 \end{proof}

\begin{proof}[Proof of Theorem~\ref{thm:alpha:3}]
Let $G$ be  a connected graph with $\alpha(G)=3$. Suppose to the contrary that $\pc(G)>4$. Then $G$ does not satisfy any condition in Lemma~\ref{prop:config}.
By Lemma~\ref{lem:folk}, we can find one of $T_1$, $T_2$, $T_3$ of Figure~\ref{fig:cases} as an induced subgraph.
When $G$ has an induced copy of $T_i$ for some $i\in[3]$,
we always follow the label in Figure~\ref{fig:cases}.
For simplicity, we let $I=\{x_1,x_2,x_3\}$, which is an independent set of $G$, $Z=V(G)\setminus V(T_i)$, and $Z_k=\{w\in Z\mid N_G(w)\cap I=\{x_k\} \}$ for every $k\in[3]$.
Clearly $Z\neq \emptyset$ and  $Z'_k=Z_k\cup\{x_k\}$ is a clique for every $k\in[3]$.

\begin{figure}[h!]
  \centering
  \includegraphics[width=11cm]{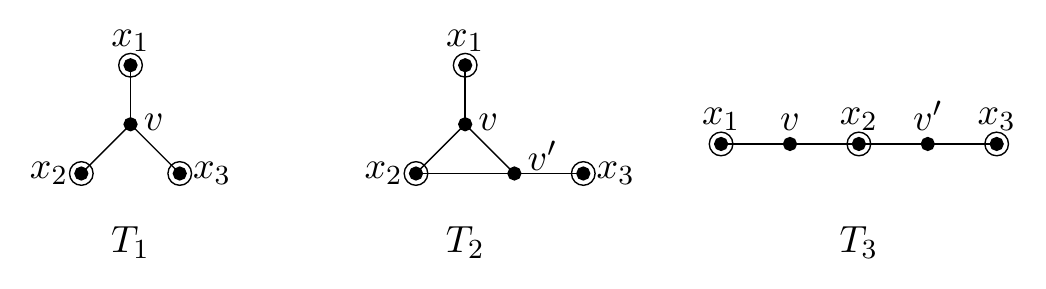}
  \caption{Induced subgraphs $T_1$-$T_3$ of $G$, where the circled vertices form an independent set of $G$}\label{fig:cases}
\end{figure}

\noindent\textbf{(Case 1)} $G$ contains $T_1$ or $T_2$ as an induced subgraph.

Suppose that $G$ contains an induced copy of $T_1$.
Since $\alpha(G)=3$, the vertex $v$ dominates $Z'_i$ for some $i\in [3]$.
Without loss of generality, we may assume that $v$ dominates $Z'_1$.
If $Z_1\cup Z_2\cup Z_3=Z$, then $V(G)$ is partitioned into three cliques $Z'_1\cup\{v\}$, $Z'_2$, $Z'_3$, which is the condition (i) of Lemma~\ref{prop:config},  a contradiction. Thus there is $w\in Z$ such that $wx_s,wx_t\in E(G)$ for some distinct $s,t\in [3]$.
Since $G$ has neither $H_5$ nor $H_6$ in Figure~\ref{fig:config} with independent set $I$,
we can conclude that $Z-(Z_1\cup Z_2\cup Z_3)=\{w\}$.
If $\{s,t\}=\{2,3\}$, then coloring $vx_2$ and $vx_3$ by the new color $2$ and $3$, respectively, gives a properly colored path between every two vertices of $G$, a contradiction.
Thus, by the symmetry of the roles of $\{1,2\}$ and $\{1,3\}$, we may assume that $\{s,t\}=\{1,2\}$.
If $Z_2=\emptyset$, then $V(G)$ can be partitioned into three cliques $Z'_1\cup\{v\}$, $\{z_2,w\}$, $Z'_3$, a contradiction.
Thus $Z_2\neq\emptyset$. We take $w'\in Z_2$.
Then the four sets $Z'_1\cup\{v\}$, $\{w,x_2\}$, $Z_2$, and $Z'_3$ form a  clique partition of $V(G)$, and the edges $x_1w$, $x_2w'$, $x_3v$,  form a matching satisfying the condition (ii) of Lemma~\ref{prop:config}, a contradiction.

Suppose that $G$ does not contain an induced copy of $T_1$. Then $G$ contains $T_2$ as an induced subgraph.
Suppose that $x_2$ has a neighbor $w\in Z$.
If $wx_1,wx_3\not\in E(G)$, then $G$ has $H_2$ in Figure~\ref{fig:config} with independent set $I$, a contradiction. Thus we may assume $wx_3\in E(G)$. Then this asserts that $G$ has $H_4$ as a subgraph with independent set $I$, a contradiction.
If $x_2$ has no neighbor in $Z$, then coloring $vx_1$ and $v'x_3$ by the new color 2 gives a properly colored path between every two vertices of $G$, a contradiction.

\medskip

\noindent\textbf{(Case 2)} $G$ contains neither $T_1$ nor $T_2$ as an induced subgraph.

Then $G$ contains $T_3$ as induced subgraph.
If there is a common neighbor $w$ of $x_1$ and $x_3$, then by the condition (iii) of Lemma~\ref{prop:config}, we may assume that $vw\in E(G)$. Then it is a contradiction since either $\{x_1,w,x_2,x_3\}$ induces $T_1$ or $\{x_1,v,w,x_2,x_3\}$ induces $T_2$. Thus there is no common neighbor of $x_1$ and $x_3$. This implies that $V(G)$ can be partitioned into three sets $Z'_1$, $Z'_3$ and $N_G[x_2]$. Clearly, $N_G[x_2]$ is not a clique since $vv'\not\in E(G)$.

\begin{claim}\label{claim:clqiue}
$N_G(x_2)$ can be partitioned into two cliques.
\end{claim}
\begin{proof}
For simplicity, let $N=N_G(x_2)\setminus \{v,v'\}$.
If $N=\emptyset$, then it is clear.
Suppose that $N \neq \emptyset$.
Let $A=\{w\in N  \mid wv\in E(G)\}$ and  $A'=\{w\in N \mid wv'\in E(G)\}$.
Since $G$ has no induced copy of $T_1$, $N =A\cup A'$.

Suppose that $w\in A$. If $wx_3\in E(G)$, then we have $H_4$ with independent set $I$, a contradiction.
Thus $wx_3\not\in E(G)$.
If $wx_1\not\in E(G)$, then we have $H_3$ with independent set $\{w,x_1,x_3\}$, a contradiction.
Thus $wx_1\in E(G)$.
Then not to have $H_4$ again, $wv'\not\in E(G)$. Hence, for every vertex $w\in A$, we have $N_G(w)\cap V(T_3)=\{x_1,v,x_2\}$.

For  $w,w'\in A$ with $w\neq w'$, if $ww'\not\in E(G)$ then $\{w,w',v',x_2\}$ induces $T_1$, a contradiction. Hence, $A\cup\{v\}$ is a clique. Similarly we can show that $A'\cup\{v'\}$ is a clique. This implies the claim.
\end{proof}

Suppose that there are $z_1\in Z_1$ and $z_3\in Z_3$ such that  $vz_1\not\in E(G)$ and $v'z_3\not\in E(G)$.
Since $\{z_1,z_3,v,v'\}$ is not an independent set of $G$, we may assume that $vz_3\in E(G)$.
Since $\{z_3,x_1,x_2\}$  is an independent set of $G$ by the definition of $Z_3$,
$\{x_1,x_2,v,z_3\}$ induces $T_1$, which is a contradiction to the case assumption.
Hence, without loss of generality, we assume that $v$ dominates $Z_1$.
Let $\{A,A'\}$ be a clique bipartition of $N_G(x_2)$ as in Claim~\ref{claim:clqiue}, say $v\in A$ and $v'\in A'$.
Then the four sets $X_1=Z'_1\cup\{v\}$, $X_2=A\cup\{x_2\}-\{v\}$, $X_3=A'$, $X_4=Z'_3$ form a clique partition of $V(G)$.
If $|X_2|=1$, that is $X_2=\{x_2\}$, then $X_1, X_2\cup X_3, X_4$ form a clique partition of $V(G)$, a contradiction in view of the condition (i) of Lemma~\ref{prop:config}. Thus $|X_2|\ge 2$, and we take a vertex $a\in X_2\setminus\{x_2\}$.

Suppose that $|X_3|=1$, that is $X_3=\{v'\}$.
If $Z_3=\emptyset$, then $X_1$, $X_2$, $\{v',x_3\}$ form a clique partition of $V(G)$, a contradiction in view of the condition (i) of Lemma~\ref{prop:config}.
Thus $Z_3\neq \emptyset$. We take a vertex $b\in Z_3$.
Then $X_1,X_2, \{v',x_3\}, Z_3$ form a clique partition of $V(G)$ and $\{va,x_2v',x_3b\}$ is a matching satisfying the condition (ii) of Lemma~\ref{prop:config}, a contradiction.
Thus $|X_3|\ge 2$, and so we can take $a'\in X_3\setminus\{v'\}$.
Then together with the four cliques $X_1,X_2,X_3,X_4$, $\{ va, x_2a', v'x_3\}$  is a matching satisfying the condition (ii) of Lemma~\ref{prop:config}, a contradiction.
\end{proof}

 \end{document}